\documentclass[a4paper,10pt]{amsart}

\usepackage{hyperref,amssymb,enumerate,url,thmtools}
\usepackage[all]{xy}

\title{The reverse mathematics of wqos and bqos}

\author{Alberto Marcone}
   \address{Dipartimento di Scienze Matematiche, Informatiche e Fisiche,
   Universit\`{a} di Udine,
   33100 Udine,
   Italy}
\email{alberto.marcone@uniud.it}
\thanks{I thank Marta Fiori Carones and Emanuele Frittaion who carefully read an earlier draft of this paper.\\
Research partially supported by PRIN 2012 Grant \lq\lq Logica, Modelli e Insiemi\rq\rq.}
\subjclass[2010]{Primary: 03B30; Secondary: 03F35, 06A07}

\newcommand{\set}[2]{\left\{\,{#1} \,:\, {#2}\,\right\}}
\newcommand{\rr}[1]{\textit{({#1})}}

\newcommand{\N}{\ensuremath{\mathbb{N}}}

\newcommand{\conc}{{{}^\smallfrown}}
\newcommand{\Sq}[1]{\left[{#1}\right]^{\N}}
\newcommand{\sq}[1]{\left[{#1}\right]^{<\N}}
\newcommand{\base}{\operatorname{base}}
\newcommand{\Pow}[1]{\mathcal{P} (#1)}
\newcommand{\Pf}[1]{\mathcal{P}_{\mathrm{f}} (#1)}
\newcommand{\init}{\sqsubset}
\newcommand{\initeq}{\sqsubseteq}
\newcommand{\tri}{\vartriangleleft}
\newcommand{\ntri}{\ntriangleleft}
\newcommand{\lh}{\operatorname{lh}}
\newcommand{\Seq}{\ensuremath{\N^{<\N}}}
\newcommand{\Qom}{Q^{<\N}}
\newcommand{\Qt}{\tilde Q}

\newcommand{\LL}{\ensuremath{\mathcal{L}_2}}

\newcommand{\PI}[2]{\ensuremath{\boldsymbol\Pi^{#1}_{#2}}}
\newcommand{\SI}[2]{\ensuremath{\boldsymbol\Sigma^{#1}_{#2}}}
\newcommand{\DE}[2]{\ensuremath{\boldsymbol\Delta^{#1}_{#2}}}

\newcommand{\RCA}{{\ensuremath{\mathsf{RCA}_0}}}

\newcommand{\WKL}{{\ensuremath{\mathsf{WKL}_0}}}

\newcommand{\ACA}{{\ensuremath{\mathsf{ACA}_0}}}
\newcommand{\ACApl}{{\ensuremath{\mathsf{ACA}_0^+}}}
\newcommand{\ACApr}{{\ensuremath{\mathsf{ACA}_0'}}}
\newcommand{\ATR}{{\ensuremath{\mathsf{ATR}_0}}}
\newcommand{\PCA}{\PI11-{\ensuremath{\mathsf{CA}_0}}}
\newcommand{\PPCA}{\PI12-{\ensuremath{\mathsf{CA}_0}}}

\newcommand{\RT}[2]{\ensuremath{\mathsf{RT}^{#1}_{#2}}}
\newcommand{\ADS}{\ensuremath{\mathsf{ADS}}}
\newcommand{\CAC}{\ensuremath{\mathsf{CAC}}}

\newcommand{\Bad}{\operatorname{Bad}}
\newcommand{\Desc}{\operatorname{Desc}}

\newcommand{\REC}{\ensuremath{\mathbf{REC}}}

\newcommand{\du}{\mathbin{\overset{\cdot}{\smash\cup}}}

\theoremstyle{plain}
\newtheorem{theorem}{Theorem}[section]
\newtheorem{lemma}[theorem]{Lemma}
\newtheorem{fact}[theorem]{Fact}
\newtheorem{question}[theorem]{Question}

\theoremstyle{definition}
\newtheorem{definition}[theorem]{Definition}

\declaretheoremstyle[notebraces = {\bf (}{)}]{stat}
\declaretheorem[numberlike=theorem, name=Statement, style=stat]{stat}

\begin{document}

\maketitle

\begin{abstract}
In this paper we survey wqo and bqo theory from the reverse mathematics
perspective. We consider both elementary results (such as the equivalence
of different definitions of the concepts, and basic closure properties) and
more advanced theorems. The classification from the reverse mathematics
viewpoint of both kinds of results provides interesting challenges, and we
cover also recent advances on some long standing open problems.
\end{abstract}

\tableofcontents\bigskip

%%%%%%%%%%%%%%%%%%%%%%%%%%%%%%%%%%%%%%%%%%%%%%%%%%%%%%%%%%%%%%%%%%%%%%%%%%%%%%%%

%\section{Introduction}\label{introduction}

This paper is an update of \cite{wqobqo}, which was written in 2000 and
documented the state of the research about the reverse mathematics of
statements dealing with wqos and bqos at the turn of the century. Since then,
new work on the subject has been carried out and we describe it here. We
however include also the results already covered by \cite{wqobqo}, attempting
to cover exhaustively the topic. We also highlight some open problems in the
area.

In Section \ref{sec:rm} we give a brief introduction to reverse mathematics
for the reader whose interest in wqos and bqos originates elsewhere. The
readers familiar with this research program can safely skip this section. In
Section \ref{sec:charwqo} we compare different characterizations of wqos and
study their closure under basic operations, such as subset, product and
intersection. Here even seemingly trivial properties provide interesting
challenges for the reverse mathematician. The study of characterizations and
closure under simple operations is repeated in Section \ref{sec:charbqo} for
bqos: the strength of some statements go all the way up to \ATR\ and
apparently simple statements such as \lq\lq$3$ is bqo\rq\rq\ have escaped
classification so far. In Section \ref{sec:mb} we consider the minimality
arguments which are one of the main proof techniques of the subject. Section
\ref{sec:struct} looks at structural results, such as the theorem by de Jongh
and Parikh asserting the existence of a maximal linear extension of a wqo.
Section \ref{sec:theorems} deals with what we might call the major results of
wqo and bqo theory, such as Higman's, Kruskal's and Nash-Williams' theorems,
the minor graph theorem and Fra\"{\i}ss\'{e}'s conjecture. We end the paper with a
section dealing with results about a topological version of wqos.

\section{Reverse mathematics}\label{sec:rm}
Reverse mathematics is a wide ranging research program in the foundations of
mathematics. The main goal of the program is to give mathematical support to
statements such as \lq\lq Theorem $A$ is stronger than Theorem $B$\rq\rq\ or
\lq\lq Theorems $C$ and $D$ are equivalent\rq\rq. If taken literally the
first statement does not make sense: since $A$ and $B$ are both true, they
are logically equivalent. By the same token, the second statement is
trivially true, and thus carries no useful information. However a
clarification of these statements is possible by finding out precisely the
minimal axioms needed to prove $B$ and showing that they do not suffice to
prove $A$, and by showing that these minimal axioms coincide for $C$ and $D$.
We are thus interested in proving equivalences between theorems and axioms,
yielding equivalences and nonequivalences between different theorems, over a
weak base theory.

Although we can label \lq\lq reverse mathematics\rq\rq\ any study of this
kind (including the study of different forms of the axiom of choice over the
base theory $\mathsf{ZF}$), the term is usually restricted to the setting of
subsystems of second order arithmetic. The language \LL\ of second order
arithmetic has variables for natural numbers and variables for sets of
natural numbers, constant symbols $0$ and $1$, binary function symbols for
addition and product of natural numbers, symbols for equality and the order
relation on the natural numbers and for membership between a natural number
and a set. A model for \LL\ consists of a first order part (an interpretation
for the natural numbers $\mathfrak{N}$ equipped with $+$, $\cdot$ and $\leq$)
and a second order part consisting of a collection of subsets of
$\mathfrak{N}$. When the first order part is standard we speak of an
$\omega$-model and we can identify the model with the subset of
$\mathcal{P}(\omega)$ that constitutes its second order part.

Second order arithmetic is the \LL-theory with classical logic consisting of
the axioms stating that the natural numbers are a commutative ordered
semiring with identity, the induction scheme for arbitrary formulas, and the
comprehension scheme for sets of natural numbers defined by arbitrary
formulas.

Hermann Weyl \cite{Weyl} and Hilbert and Bernays \cite{HB1,HB2} already
noticed in their work on the foundations of mathematics that \LL\ is rich
enough to express, using appropriate codings, significant parts of
mathematical practice, and that many mathematical theorems are provable in
(fragments of) second order arithmetic. Actually Weyl used a theory similar
to what we now denote by \ACApl\ (a slight strengthening of \ACA, to be
described below). Recently Dean and Walsh \cite{prehistory} traced the
history of subsystems of second order arithmetic leading to \cite{Frie},
where Harvey Friedman started the systematic search for the axioms that are
sufficient and necessary to prove theorems of ordinary, not set-theoretic,
mathematics. One of Friedman's main early discoveries was that (in his words)
\lq\lq When the theorem is proved from the right axioms, the axioms can be
proved from the theorem\rq\rq. Friedman also highlighted the role of
set-existence axioms, and this soon led to restricting the induction
principles allowed in the various systems. The base system \RCA\ and the now
well-known \WKL, \ACA, \ATR, and \PCA, were introduced in \cite{Frirestr}.
Today, most of reverse mathematics research compares the strength of
mathematical theorems by establishing equivalences, implications and
nonimplications over \RCA.

To describe \RCA\ and the other systems used in reverse mathematics let us
also recall that formulas of \LL\ are classified in the usual hierarchies:
those with no set quantifiers and only bounded number quantifiers are \DE00,
while counting the number of alternating unbounded number quantifiers we
obtain the classification of all arithmetical (= without set quantifiers)
formulas as \SI0n and \PI0n formulas (one uses \SI{}{} or \PI{}{} depending
on the type of the first quantifier in the formula, existential in the
former, universal in the latter). Formulas with set quantifiers in front of
an arithmetical formula are classified by counting their alternations as
\SI1n and \PI1n. A formula is \DE in in a given theory if it is equivalent in
that theory both to a \SI in formula and to a \PI in formula.

In \RCA\ the induction scheme and the comprehension scheme of second order
arithmetic are restricted respectively to \SI01 and \DE01 formulas. \RCA\ is
strong enough to prove some basic results about many mathematical structures,
but too weak for many others. The $\omega$-models of \RCA\ are the Turing
ideals: subsets of $\mathcal{P}(\omega)$ closed under join and Turing
reducibility. The minimal $\omega$-model of \RCA\ consists of the computable
sets and is usually denoted by \REC.

If a theorem $T$ is expressible in \LL\ but unprovable in \RCA, the reverse
mathematician asks the question: what is the weakest axiom we can add to
\RCA\ to obtain a theory that proves $T$? In principle, we could expect that
this question has a different answer for each $T$, but already Friedman
noticed that this is not the case. In fact, most theorems of ordinary
mathematics expressible in \LL\ are either provable in \RCA\ or equivalent
over \RCA\ to one of the following four subsystems of second order
arithmetic, listed in order of increasing strength: \WKL, \ACA, \ATR, and
\PCA. This is witnessed in Steve Simpson's monograph \cite{sosoa} and
summarized by the \emph{Big Five} terminology. We thus obtain a neat picture
where theorems belonging to quite different areas of mathematics are
classified in five levels, roughly corresponding to the mathematical
principles used in their proofs. \RCA\ corresponds to \lq\lq computable
mathematics\rq\rq, \WKL\ embodies a compactness principle, \ACA\ is linked to
sequential compactness, \ATR\ allows for transfinite arguments, \PCA\
includes impredicative principles.

To obtain \WKL\ we add to \RCA\ the statement of Weak K\"{o}nig's Lemma, i.e.,
every infinite binary tree has a path, which is essentially the compactness
of Cantor space. An equivalent statement, intuitively showing that \WKL\ is
stronger than \RCA\ (a rigorous proof needs simple arguments from model
theory and computability theory), is \SI01-separation: if $\varphi(n)$ and
$\psi(n)$ are \SI01-formulas such that $\forall n\, \neg (\varphi(n) \land
\psi(n))$ then there exists a set $X$ such that $\varphi(n) \implies n \in X$
and $\psi(n) \implies n \notin X$ for all $n$. \WKL\ and \RCA\ have the same
consistency strength of Primitive Recursive Arithmetic, and are thus
proof-theoretically fairly weak. Nevertheless, \WKL\ proves (and often turns
out to be equivalent to) a substantial amount of classical mathematical
theorems, including many results about real-valued functions and countable
rings and fields, basic Banach space facts, etc. The $\omega$-models of \WKL\
are the Scott ideals, and their intersection consists of the computable sets.

\ACA\ is obtained from \RCA\ by extending the comprehension scheme to all
arithmetical formulas. The statements without set variables provable in \ACA\
coincide exactly with the theorems of Peano Arithmetic, so that in particular
the consistency strength of the two theories is the same. Within \ACA\ one
can develop a fairly extensive theory of continuous functions, using the
completeness of the real line as an important tool. \ACA\ proves (and often
turns out to be equivalent to) also many basic theorems about countable
fields, rings, and vector spaces. For example, \ACA\ is equivalent, over
\RCA, to the Bolzano-Weierstrass theorem on the real line. The
$\omega$-models of \ACA\ are the Turing ideals closed under jumps, so that
the minimal $\omega$-model of \ACA\ consists of all arithmetical sets.

\ATR\ is the strengthening of \RCA\ (and \ACA) obtained by allowing to
iterate arithmetical comprehension along any well-order. It can be shown
\cite[Theorem V.5.1]{sosoa} that, over \RCA, \ATR\ is equivalent to
\SI11-separation, which is exactly as \SI01-separation but with \SI11
formulas allowed. This is a theory at the outer limits of predicativism and
proves (and often turns out to be equivalent to) many basic statements of
descriptive set theory but also some results from advanced algebra, such as
Ulm's theorem.

\PCA\ is the strongest of the big five systems, and is obtained from \RCA\ by
extending the comprehension scheme to \PI11 formulas. Also this axiom scheme
is equivalent to many results, including some from descriptive set theory,
Banach space theory and advanced algebra, such as the structure theorem for
countable Abelian groups.\smallskip

In recent years there has been a change in the reverse mathematics main
focus: following Seetapun's breakthrough result that Ramsey theorem for pairs
is not equivalent to any of the Big Five systems \cite{See}, a plethora of
statements, mostly in countable combinatorics, have been shown to form a rich
and complex web of implications and nonimplications. The first paper
featuring complex and non-linear diagrams representing the relationships
between statements of second order arithmetics appears to be \cite{HS}.
Nowadays diagrams of this kind are a common feature of reverse mathematics
papers. This leads to the zoo of reverse mathematics, a terminology coined by
Damir Dzhafarov when he designed \lq\lq a program to help organize relations
among various mathematical principles, particularly those that fail to be
equivalent to any of the big five subsystems of second-order
arithmetic\rq\rq. Hirschfeldt's monograph \cite{slicing} highlights this new
focus of the reverse mathematics program.

Many elements of the zoo are connected to Ramsey theorem. By \RT k\ell\ we
denote Ramsey theorem for sets of size $k$ and $\ell$ colors: for every
coloring $c: [\N]^k \to \ell$ (here $[X]^k$ is the set of all subsets of $X$
with exactly $k$ elements, and $\ell$ is the set $\{0, , \dots, \ell-1\}$)
there exists an infinite homogenous set $H$, i.e., such that for some
$i<\ell$ we have $c(s)=i$ for every $s \in [H]^k$. \RT k{<\infty} is $\forall
\ell\, \RT k\ell$. A classic result is that \RT k\ell\ is equivalent to \ACA\
over \RCA\ when $k \geq 3$ and $\ell \geq 2$ (see \cite[\S III.7]{sosoa}). On
the other hand, building on Seetapun's result with the essential new step
provided by Liu \cite{Liu}, we now know that \RT22 and \RT2{\infty} are both
incomparable with \WKL\ (see \cite[\S6.2 and Appendix]{slicing}). For any
fixed $\ell$ the infinite pigeonhole principle for $\ell$ colors \RT1\ell\ is
provable in \RCA. On the other hand the full infinite pigeonhole principle
\RT1{<\infty} is not provable in \RCA\ and not even in \WKL; in fact it is
equivalent over \RCA\ to the principle known as \SI02-bounding, which is
intermediate in strength between \SI01-induction and \SI02-induction.

Two of the earliest examples of the zoo phenomenon play a significant role
with respect to statements dealing with wqos. Both statements are fairly
simple consequences of \RT22. \CAC\ is the statement that any infinite
partial order contains either an infinite antichain or an infinite chain,
while \ADS\ asserts that every infinite linear order has either an infinite
ascending chain or an infinite descending chain. Hirschfeldt and Shore
\cite{HS} showed that \RT22 is properly stronger than \CAC, which in turn
implies \ADS. They also showed that none of these principles imply \WKL\ over
\RCA. The fact that \CAC\ is properly stronger than \ADS\ was first proved by
Lerman, Solomon, and Towsner \cite{LST}, and then given a simpler proof by
Patey \cite{Patey2016}. These results support the idea that \RT22, in
contrast to the big five, is not robust (Montalb\'{a}n \cite{Mont} informally
defined a theory to be robust \lq\lq if it is equivalent to small
perturbations of itself\rq\rq).\smallskip

Wqo and bqo theory represents an area of combinatorics which has always
interested logicians. From the viewpoint of reverse mathematics, one of the
reasons for this interest stems from the fact that some important results
about wqos and bqos appear to use axioms that are within the realm of second
order arithmetic, yet are much stronger than those necessary to develop other
areas of ordinary mathematics (as defined in the introduction of
\cite{sosoa}). We will see that results about wqo and bqo belong to both
facets of reverse mathematics: some statements fit neatly in the big five
picture, while some others provide examples of the zoo.

When dealing with wqo and bqo theory, at first sight the limitations of the
expressive power of second-order arithmetic compel us to consider only
quasi-orders defined on countable sets. This is actually not a big
restriction because a quasi-order is wqo (resp.\ bqo) if and only if each of
its restrictions to a countable subset of its domain is wqo (resp.\ bqo). The
limitation mentioned above must be adhered to when we quantify over the
collection of all wqos (or bqos), typically in statements of the form \lq\lq
for every wqo \dots \rq\rq. However we can also consider specific
quasi-orders defined on uncountable sets (such as the powerset of a countable
set, the collection of infinite sequences of elements of a countable set, or
the set of all countable linear orders); statements about these (with a fixed
quasi-order) being wqo or bqo can be expressed in a natural way in
second-order arithmetic (see Definition \ref{uncountable} below).\smallskip

We often use $\leq_\N$ for the order relation given by the symbol $\leq$ in
the language of second order arithmetic. This notation helps to emphasize
when we are comparing elements of a quasi-order via the quasi-order relation
and when we are comparing them via the underlying structure of arithmetic. We
use this notation when the distinction between these orders is not
immediately clear from the context.

As usual in the reverse mathematics literature, whenever we begin a
definition or statement with the name of a subsystem of second order
arithmetic in parenthesis we mean that the definition is given, or the
statement proved, within that subsystem.

\section{Characterizations and basic properties of wqos}\label{sec:charwqo}

\begin{definition}[\RCA]
A \emph{quasi-order} is a pair $(Q, {\preceq})$ such that $Q$ is a set and
$\preceq$ is a transitive reflexive relation on $Q$.
\end{definition}

When there is no danger of confusion we assume that $Q$ is always equipped
with the quasi-order $\preceq$ and that $\preceq$ is always a quasi-order on
the set $Q$. Thus in our statements we often mention only $\preceq$ or only
$Q$.

Partial orders are natural examples of quasi-orders: a partial order is a
quasi-order which also satisfies antisymmetry. We can transform a quasi-order
$Q$ into a partial order using the equivalence relation defined by $x \sim y$
if and only if $x \preceq y$ and $y \preceq x$. The quotient structure
$Q/\!{\sim}$ is naturally equipped with a partial order which can be formed
using $\mathbf{\Delta}^0_1$ comprehension in \RCA\ (it suffices to identify
an equivalence class with its least member with respect to $\leq_\N$).

Much of the standard terminology and notation for partial orders is used also
when dealing with quasi-orders. For example, we write $x \perp y$ to indicate
that $x$ and $y$ are incomparable under $\preceq$ and we write $x \prec y$ if
$x \preceq y$ and $y \npreceq x$.

\begin{definition}[\RCA]
A set $A \subseteq Q$ is an \emph{antichain} if $x \perp y$ for all $x \neq y
\in A$. A set $C \subseteq Q$ is a \emph{chain} if $x \preceq y$ or $y
\preceq x$ for all $x, y \in C$.

A set $I \subseteq Q$ is an \emph{initial interval} if $y \in I$ whenever $y
\preceq x$ for some $x \in I$. The definition of \emph{final interval} is
symmetric, with $x \preceq y$ for some $x \in I$.
\end{definition}

%It is easy to see (and prove within \RCA) that any antichain (resp.\ chain,
%initial interval, final interval) in $Q$ gives rise to a corresponding
%antichain (resp.\ chain, initial interval, final interval) in $Q/\!{\sim}$
%and vice versa. This allows us to work with partial orders rather than
%quasi-orders whenever it is more convenient.

\begin{definition}[\RCA]
A quasi-order $(Q,\preceq)$ is \emph{linear} if $Q$ is a chain.

If $\preceq$ is a quasi-order on $Q$ and $\preceq_L$ is a linear quasi-order
on $Q$, then we say $\preceq_L$ is a \emph{linear extension} of $\preceq$ if
for all $x,y \in Q$, $x \preceq y$ implies $x \preceq_L y$ and $x \sim_L y$
implies $x \sim y$.
\end{definition}

Notice that (provably in \RCA) if $Q$ is a linear quasi-order then
$Q/\!{\sim}$ is a linear order. Moreover, if $\preceq_L$ is a linear
extension of $\preceq$ then $x \sim y$ if and only if $x \sim_L y$ and
therefore the linear extensions of a quasi-order $Q$ correspond exactly to
the linear extensions of the partial order $Q/\!{\sim}$.\smallskip

We can now give the official definition of wqo within \RCA.

\begin{definition}[\RCA]
Let $\preceq$ be a quasi-order on $Q$. $(Q,{\preceq})$ is \emph{wqo} if for
every map $f: \N \to Q$ there exist $m <_\N n$ such that $f(m) \preceq f(n)$.
\end{definition}

\begin{definition}[\RCA]
An infinite sequence of elements of $Q$ is a function $f: A \to Q$ where $A
\subseteq \N$ is infinite.

$f$ is \emph{ascending} if $f(n) \prec f(m)$ for all $n,m \in A$ with $n <_\N
m$. Similarly, $f$ is \emph{descending} if $f(m) \prec f(n)$ whenever $n,m
\in A$ are such that $n <_\N m$.

A \emph{well-order} is a linear quasi-order with no infinite descending
sequences.

We say that $f$ is a \emph{good sequence} (with respect to $\preceq$) if
there exist $m,n \in A$ such that $m <_\N n$ and $f(m) \preceq f(n)$; if this
does not happen we say that $f$ is \emph{bad}.
\end{definition}

The following characterization of wqo is immediate, and easy to prove within
\RCA\ using the existence of the enumeration of the elements of an infinite
subset of \N\ in increasing order:

\begin{fact}[\RCA]\label{charactwqo}
Let $(Q,{\preceq})$ be a quasi-order. The following are equivalent:
\begin{enumerate}[\quad (i)]
  \item $Q$ is wqo;
  \item every sequence of elements of $Q$ is good with respect to
      $\preceq$.
\end{enumerate}\end{fact}

Wqos can be characterized by several other statements about quasi-orders. The
systematic investigation of the axioms needed to prove the equivalences
between these characterizations was started by Cholak, Marcone, and Solomon
in \cite{defwqo}.

Let us begin with the characterizations which are provable in \RCA.

\begin{lemma}[\RCA]
Let $(Q,{\preceq})$ be a quasi-order. The following are equivalent:
\begin{enumerate}[\quad (i)]
  \item $Q$ is wqo;
  \item $Q$ has the finite basis property, i.e., for every $X \subseteq Q$
      there exists a finite $F \subseteq X$ such that $\forall x \in X\,
      \exists y \in F\, y \preceq x$;
  \item there is no infinite sequence of initial segments of $Q$ which is
      strictly decreasing with respect to inclusion;
  \item there is no infinite sequence of final segments of $Q$ which is
      strictly increasing with respect to inclusion.
\end{enumerate}\end{lemma}

The equivalence between \rr{i} and \rr{ii} was already noticed by Simpson
(see \cite[Lemma 3.2]{Simpson88b}, where the finite basis property is stated
in terms of partial orders rather than quasi-orders: full details with the
current definition are provided in \cite[Lemma 4.8]{wqobqo}). The equivalence
between \rr{iii} and \rr{iv} is immediate by taking complements with respect
to $Q$. To show that \rr{i} implies \rr{iii} start from an infinite sequences
$\set{I_n}{n \in \N}$ of initial segments of $Q$ which is strictly decreasing
with respect to inclusion and for every $n$ let $f(n)$ be the $\leq_\N$
minimum element of $I_n \setminus I_{n+1}$: $f$ is a bad sequence. To prove
that \rr{iii} implies \rr{i} let $f$ be a bad sequence with domain \N\ and
set $I_n = \set{x \in Q}{\forall i \leq n\, f(i) \npreceq x}$: $\set{I_n}{n
\in \N}$ is an infinite strictly decreasing sequence of initial segments of
$Q$.

We now consider the characterizations of the notion of wqo which turn out to
be more interesting from the reverse mathematics viewpoint.

\begin{definition}[\RCA]
Let $(Q,{\preceq})$ be a quasi-order:
\begin{itemize}
\item $Q$ is \emph{wqo(set)} if for every $f:\mathbb{N} \rightarrow Q$
    there is an infinite set $A$ such that for all $n,m \in A$, $n <_\N m
    \rightarrow f(n) \preceq f(m)$;
\item $Q$ is \emph{wqo(anti)} if it has no infinite descending sequences
    and no infinite antichains;
\item $Q$ is \emph{wqo(ext)} if every linear extension of $\preceq$ is a
    well-order.
\end{itemize}
\end{definition}

\RCA\ proves quite easily some implications: every wqo(set) is wqo, and every
wqo is both wqo(anti) and wqo(ext). Cholak, Marcone, and Solomon showed that
all other implications between these notions are not true in the
$\omega$-model \REC, and hence are not provable within \RCA.

\begin{theorem}\label{defRCA}
The implications between the notions of wqo, wqo(set), wqo(anti) and wqo(ext)
which are provable in \RCA\ are exactly the ones in the transitive closure of
the diagram:
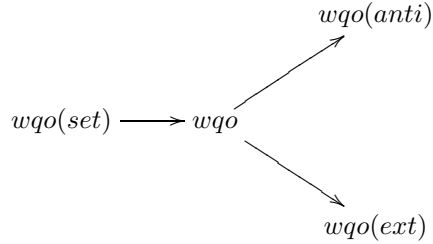
\[
\xymatrix{ & & wqo(anti)\\
wqo(set) \ar[r] & wqo \ar[ur] \ar[dr]\\
& & wqo(ext)}
\]
In fact the above diagram depicts the implications which hold in \REC, and
thus adding induction axioms to \RCA\ yields no other implications.
\end{theorem}

To show that every wqo implies wqo(set) fails in \REC, it suffices to recall
a classical construction (due to Denisov and Tennenbaum independently: see
\cite{Downey}) of a computable linear order of order type $\omega +
\omega^{\ast}$ which does not have any infinite computable ascending or
descending sequences.

Similarly, showing that \REC\ does not satisfy that every wqo(ext) is wqo
means building a computable partial order $(Q,{\preceq})$ such that all its
computable linear extensions are computably well-ordered (i.e., do not have
infinite computable descending sequences) but there is a computable $f: \N
\to Q$ such that $f(m) \npreceq f(n)$ for all $m <_\N n$. In fact the partial
order constructed in \cite[Theorem 3.21]{defwqo} using a finite injury
construction is such that $f(m) \perp f(n)$ for all $m \neq n$, thus
obtaining the stronger result that \REC\ does not satisfy that every wqo(ext)
is wqo(anti).

To show that wqo(anti) implies wqo does not hold in \REC\ one needs to find a
computable partial order $(Q,{\preceq})$ with no computable infinite
antichains and no computable infinite descending sequences but such that
there exists a computable $f: \N \to Q$ such that $f(m) \npreceq f(n)$ for
all $m <_\N n$. The partial order built in \cite[Theorem 3.9]{defwqo} has the
additional property of having a computable linear extension with a computable
infinite descending sequence (see \cite[Corollary 3.10]{defwqo}). Hence \REC\
does not satisfy that every wqo(anti) is wqo(ext).

One can improve the latter construction obtaining even more information. In
fact, \cite[Theorem 3.11]{defwqo} shows that if $(X_i)_{i \in \N}$ is a
sequence of uniformly $\Delta^0_2$, uniformly low sets there exists a
computable partial order $(Q,{\preceq})$, such that for all $i$ no
$X_i$-computable function lists an infinite antichain or an infinite
descending sequence in $Q$, but there exists a computable $f: \N \to Q$ such
that $f(m) \npreceq f(n)$ for all $m <_\N n$. Since for an appropriate choice
of $(X_i)_{i \in \N}$ we have that the $\omega$-model $\set{Y}{\exists i (Y
\leq_T X_i)}$ satisfies \WKL, we obtain that \WKL\ does not prove that every
wqo(anti) is wqo.

Further exploring the provability of the other implications in \WKL, we
notice that it is fairly easy to prove in \RCA\ that the statement that every
wqo is wqo(set) implies \RT1{<\infty} (\cite[Lemma 3.20]{defwqo}), and hence
is not provable in \WKL.

On the other hand, \cite[Theorem 3.17]{defwqo} shows that \WKL\ proves (using
the fact, equivalent to \WKL, that every acyclic relation is contained in a
partial order) that every wqo(ext) is wqo. Putting the information mentioned
above together we obtain the following picture regarding provability in \WKL.

\begin{theorem}\label{defWKL}
The implications between the notions of wqo, wqo(set), wqo(anti) and wqo(ext)
which are provable in \WKL\ are exactly the ones in the transitive closure of
the diagram:
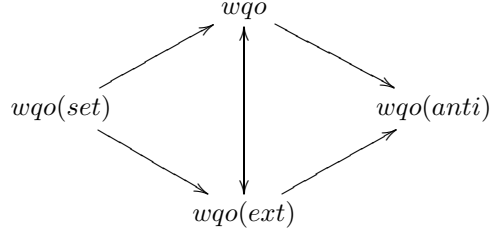
\[
\xymatrix{ & wqo \ar[dr] \ar@{<->}[dd]\\
wqo(set) \ar[ur] \ar[dr] & & wqo(anti)\\
& wqo(ext) \ar[ur]}
\]
\end{theorem}

This leads to the following natural question, which has resisted any attempt
so far.

\begin{question}\label{q:WKL}
Consider the statements \lq\lq every wqo(ext) is wqo\rq\rq\ and \lq\lq every
wqo(ext) is wqo(anti)\rq\rq. Are they equivalent to \WKL\ over \RCA?
\end{question}

On the other hand, the statement \lq\lq every wqo(anti) is wqo(set)\rq\rq\
turns out to be equivalent to \CAC\ over \RCA\ (\cite[Lemma 3.3]{defwqo}). It
follows that the statements \lq\lq every wqo(anti) is wqo\rq\rq\ and \lq\lq
every wqo is wqo(set)\rq\rq\ are both provable from \CAC.

\begin{theorem}\label{defCAC}
$\RCA + \CAC$ proves the implications between the notions of wqo, wqo(set),
wqo(anti) and wqo(ext) which are in the transitive closure of the diagram:
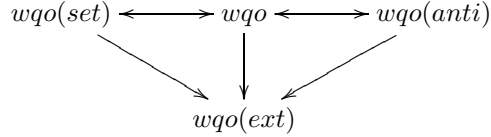
\[
\xymatrix{ wqo(set) \ar@{<->}[r] \ar[dr] & wqo \ar@{<->}[r]
\ar[d] & wqo(anti) \ar[dl]\\
& wqo(ext)}
\]
\end{theorem}

The diagram of Theorem \ref{defCAC} is different from the ones of Theorems
\ref{defRCA} and \ref{defWKL} in that it is unknown whether the missing
implications can be proved in $\RCA + \CAC$.

\begin{question}\label{q:CAC}
Does $\RCA + \CAC$ proves \lq\lq every wqo(ext) is wqo\rq\rq?
\end{question}

Notice that a positive answer to Question \ref{q:WKL} implies, since $\RCA +
\CAC$ does not prove \WKL, a negative answer to Question \ref{q:CAC}.\smallskip
% To
%obtain a negative answer to the latter question one could also try to build
%on Liu's result (\cite{Liu}, see also \cite{slicing}) that \RT22 does not
%imply \WKL\ to show that \RT22 (and hence \CAC) does not imply \lq\lq every
%wqo(ext) is wqo\rq\rq.

\RCA\ easily proves that all well-orders and all finite quasi-orders are wqo
(indeed for the latter fact the finite pigeonhole principle suffices). By
Theorem \ref{defRCA} the same happens for wqo(anti) and wqo(ext). Regarding
wqo(set) we have that, using the appropriate \RT1\ell, for any specific
finite quasi-order \RCA\ proves that the quasi-order is wqo(set). On the
other hand, it is not difficult to see that, over \RCA, \lq\lq every finite
quasi-order is wqo(set)\rq\rq\ is equivalent to \RT1{<\infty}, while \lq\lq
every well-order is wqo(set)\rq\rq\ is equivalent to \ADS.\smallskip

Wqos enjoy several basic closure properties. The study of these from the
viewpoint of reverse mathematics was started in \cite{wqobqo} and
\cite{defwqo}.

We first consider the basic property of closure under taking subsets. The
proof of the following lemma is immediate.

\begin{lemma}[\RCA]\label{RCAsubset}
Let $\mathcal{P}$ be any of the properties wqo, wqo(anti) or wqo(set). If
$(Q,{\preceq})$ satisfies $\mathcal{P}$ and $R \subseteq Q$ then the
restriction of $\preceq$ to $R$ satisfies $\mathcal{P}$ as well.
\end{lemma}

If $\mathcal{P}$ is wqo(ext) then the statement of Lemma \ref{RCAsubset} is
slightly more difficult to prove, since the obvious proof of the reversal is
based on the following fact: if $(Q, {\preceq})$ is a partial order, $R
\subseteq Q$, $\preceq_L$ is a linear extension of the restriction of
$\preceq$ to $R$, then there exists a linear extension of the whole $\preceq$
which extends also $\preceq_L$. \WKL\ suffices to prove this statement,
because we can consider ${\preceq} \cup {\preceq_L}$, which is an acyclic
relation, extend it to a partial order (here is the step using \WKL, see
\cite[Lemma 3.16]{defwqo}), and then to a linear order (\RCA\ suffices for
this last step).

\begin{question}\label{wqo(ext)subset}
Does \RCA\ suffice to prove that if $(Q,{\preceq})$ is wqo(ext) and $R
\subseteq Q$ then the restriction of $\preceq$ to $R$ is also wqo(ext)? Is
this implication equivalent to \WKL?
\end{question}

Let us now consider basic closure operations that involve two quasi-orders.

\begin{definition}[\RCA]\label{cartprod}
If $\preceq_1$ and $\preceq_2$ are quasi-orders on $Q_1$ and $Q_2$ we may
assume that $Q_1 \cap Q_2 = \emptyset$ (or replace each $Q_i$ by its
isomorphic copy on $Q_i \times \{i\}$). We can define the \emph{sum
quasi-order} and the \emph{disjoint union quasi-order} on $Q_1 \cup Q_2$
(denoted by $Q_1 + Q_2$ and by $Q_1 \du Q_2$ respectively) by
\begin{align*}
    x \preceq_+ y & \iff (x \in Q_1 \land y \in Q_2) \lor (x,y \in Q_1 \land x \preceq_1 y) \lor (x,y \in Q_2 \land x \preceq_2 y);\\
    x \preceq_{\du} y & \iff (x,y \in Q_1 \land x \preceq_1 y) \lor (x,y \in Q_2 \land x \preceq_2 y).
\end{align*}
The \emph{product quasi-order} on $Q_1 \times Q_2$ is defined by
\[
(x_1,x_2) \preceq_\times (y_1,y_2) \iff x_1 \preceq_1 y_1 \land x_2 \preceq_2 y_2.
\]
Moreover if $\preceq_1$ and $\preceq_2$ are quasi-orders on the same set $Q$
then the \emph{intersection quasi-order} on $Q$ is defined by
\[
x \preceq_{\cap} y \iff x \preceq_1 y \land x \preceq_2 y.
\]
\end{definition}

The following lemma follows easily from the provability in \RCA\ of \RT12.

\begin{lemma}[\RCA]\label{+duwqo}
Let $\mathcal{P}$ be any of the properties wqo, wqo(ext), wqo(anti) or
wqo(set). If $Q_1$ and $Q_2$ satisfy $\mathcal{P}$ then $Q_1 + Q_2$ and $Q_1
\du Q_2$ satisfy $\mathcal{P}$ with respect to the sum and disjoint union
quasi-orders.
\end{lemma}

The next lemma was first noticed in \cite{wqobqo} for wqos, and then extended
to the other notions in \cite{defwqo}.

\begin{lemma}[\RCA]\label{cup=times}
Let $\mathcal{P}$ be any of the properties wqo, wqo(anti) or wqo(set). The
following are equivalent:
\begin{enumerate}[\quad (i)]
    \item if $Q$ satisfies $\mathcal{P}$ with respect to the quasi-orders
        $\preceq_1$ and $\preceq_2$ then $Q$ satisfies $\mathcal{P}$ with
        respect to the intersection quasi-order;
    \item if $Q_1$ and $Q_2$ satisfy $\mathcal{P}$ then $Q_1 \times Q_2$
        satisfies $\mathcal{P}$ with respect to the product quasi-order.
\end{enumerate}
\end{lemma}

The proof of \rr{i} implies \rr{ii} is based on the fact that products can be
realized as intersections and works for wqo(ext) as well. The proof of
\rr{ii} implies \rr{i} uses the fact that intersections can be viewed as
subsets of products, and thus employs Lemma \ref{RCAsubset}. In \cite{defwqo}
it is claimed that Lemma \ref{cup=times} holds also when $\mathcal{P}$ is
wqo(ext), but it seems that this might depend on the answer to Question
\ref{wqo(ext)subset}.

\begin{question}
Let $\mathcal{P}$ be wqo(ext). Does \RCA\ suffice to prove that \rr{ii} of
Lemma \ref{cup=times} implies \rr{i}? Is this implication equivalent to \WKL?
\end{question}

The following results are from \cite{defwqo}.

\begin{lemma}[\RCA]\label{wqo(set)prod}
Let $\mathcal{P}$ be any of the properties wqo, wqo(ext), wqo(anti) or
wqo(set).
\begin{itemize}
\item If $Q$ is wqo(set) with respect to the quasi-orders $\preceq_1$ and
    $\preceq_2$ then $Q$ satisfies $\mathcal{P}$ with respect to the
    intersection quasi-order;
\item if $Q_1$ and $Q_2$ are wqo(set) then $Q_1 \times Q_2$ satisfies
    $\mathcal{P}$ with respect to the product quasi-order.
\end{itemize}\end{lemma}

\begin{theorem}\label{wqoprod}
Let $\mathcal{P}_1$ be any of the properties wqo, wqo(ext) and wqo(anti). Let
$\mathcal{P}_2$ be any of the properties wqo, wqo(set), wqo(ext) and
wqo(anti).
\begin{itemize}
\item \WKL\ does not prove that if $Q$ satisfies $\mathcal{P}_1$ with
    respect to the quasi-orders $\preceq_1$ and $\preceq_2$ then $Q$
    satisfies $\mathcal{P}_2$ with respect to the intersection quasi-order;
\item \WKL\ does not prove that if $Q_1$ and $Q_2$ satisfy $\mathcal{P}_1$
    then $Q_1 \times Q_2$ satisfies $\mathcal{P}_2$ with respect to the
    product quasi-order.
\end{itemize}
\end{theorem}

All instances of Theorem \ref{wqoprod} follow easily (using Lemma
\ref{cup=times} and Theorem \ref{defWKL}) from Theorem 4.3 of \cite{defwqo}.
To state this theorem fix an $\omega$-model $\mathcal{M}$ of \WKL\ which
consists of the sets Turing reducible to a member of a sequence of uniformly
$\Delta^0_2$, uniformly low sets. The theorem asserts the existence of
computable partial orders $\preceq_0$ and $\preceq_1$ which are wqo in
$\mathcal{M}$ (i.e., $\mathcal{M}$ contains no bad sequence with respect to
either $\preceq_0$ or $\preceq_1$) and such that ${\preceq_0} \cap
{\preceq_1}$ is an infinite antichain (so that the intersection is not
wqo(anti)). The construction of $\preceq_0$ and $\preceq_1$ is by a finite
injury argument.

Theorem \ref{defCAC} and Lemma \ref{wqo(set)prod} imply that $\RCA + \CAC$
proves the closure of wqos under product. On the other hand Frittaion,
Marcone, and Shafer pointed out that this statement implies \ADS\ and asked
for a classification. Recently, Henry Towsner \cite{Towprod} gave a typical
zoo answer to this question by proving the following theorem.

\begin{theorem}
\WKL\ does not prove that the closure of wqos under product implies \CAC, nor
that \ADS\ implies the closure of wqos under product.
\end{theorem}

Towsner starts by translating the statement in Ramsey-theoretic terms. Given
the coloring $c:[\N]^2 \to \ell$ we say that color $i$ is \emph{transitive}
if $c(k_0,k_2)=i$ whenever $c(k_0,k_1) = c(k_1,k_2) = i$ for some $k_1$
satisfying $k_0 < k_1 < k_2$. Hirschfeldt and Shore \cite{HS} noticed that
\ADS\ is equivalent to the restriction of \RT22 to colorings such that both
colors are transitive, while \CAC\ is equivalent to the restriction of \RT22
to colorings with one transitive color. Towsner notices that the closure of
wqos under product is equivalent to the following intermediate statement: if
$c: [\N]^2 \to 3$ is such that colors $0$ and $1$ are transitive then there
exists an infinite set $H$ such that for some $i<2$ we have $c(s) \neq i$ for
every $s \in [H]^k$ (i.e., $H$ avoids one of the transitive colors). Then he
proceeds to construct Scott ideals with the appropriate properties: the first
satisfies the above transitive color avoiding statement but not the
restriction of \RT22 to colorings with one transitive color; the second
satisfies for all $\ell$ the restriction of \RT2\ell\ to colorings such that
all color are transitive, but fails to satisfy the statement equivalent to
the closure of wqos under product.

Special instances of the closure of wqos under product have been studied by
Simpson \cite{Simpson88b}.

\begin{theorem}[\RCA]\label{finitepowersN}
Let $\omega$ denote the order $(\N, {\leq_\N})$. Then
\begin{enumerate}[1.]
\item the product of two copies of $\omega$ is wqo with respect to the
    product quasi-order.
\item the following are equivalent:%
\begin{enumerate}[\quad (i)]
  \item $\omega^\omega$ is well-ordered;
  \item for every $k \in \N$ the product of $k$ copies of $\omega$ is wqo
      with respect to the (obvious generalization of the) product
      quasi-order.
\end{enumerate}\end{enumerate}\end{theorem}

Since $\omega^\omega$ is the proof theoretic ordinal of \RCA, it follows that
\RCA\ does not prove the statement \rr{ii} above.

Recently Hatzikiriakou and Simpson \cite{HatzSim} proved that another
statement dealing with wqos is equivalent to the fact that $\omega^\omega$ is
well-ordered. A Young diagram is a sequence of natural numbers $\langle m_0,
\dots, m_k \rangle$ such that $m_i \geq m_{i+1}$ and $m_k>0$. We denote by
$\mathcal{D}$ the set of all Young diagrams, and set $\langle m_0, \dots, m_k
\rangle \preceq_\mathcal{D} \langle n_0, \dots, n_h \rangle$ if and only if
$k \leq h$ and $m_i \leq n_i$ for all $i\leq k$.

\begin{theorem}[\RCA]\label{Young}
The following are equivalent:
\begin{enumerate}[\quad (i)]
  \item $\omega^\omega$ is well-ordered;
  \item $(\mathcal{D}, {\preceq_\mathcal{D}})$ is wqo.
\end{enumerate}
\end{theorem}

Theorems \ref{finitepowersN} and \ref{Young} are both motivated by the study
of results about the non-existence of infinite ascending sequences of ideals
in rings.

\section{Characterizations and basic properties of bqos}\label{sec:charbqo}
To give the definition of bqo we need some terminology and notation for
sequences and sets (here we follow \cite{NWT}). All the definitions are given
in \RCA. Let \Seq\ be the set of finite sequences of natural numbers. If $s
\in \Seq$ we denote by $\lh s$ its length and, for every $i< \lh s$, by
$s(i)$ its ($i+1$)-th element. Then we write this sequence as $s = \langle
s(0), \dots, s(\lh s -1) \rangle$. If $s,t \in \Seq$ we write $s \initeq t$
if $s$ is an initial segment of $t$, i.e., if $\lh s \leq \lh t$ and $\forall
i < \lh s\, s(i)=t(i)$. We write $s \subseteq t$ if the range of $s$ is a
subset of the range of $t$, i.e., if $\forall i < \lh s\, \exists j < \lh t\,
s(i) = t(j)$. $s \init t$ and $s \subset t$ have the obvious meanings. We
write $s \conc t$ for the concatenation of $s$ and $t$, i.e., the sequence
$u$ such that $\lh u = \lh s + \lh t$, $u(i) = s(i)$ for every $i<\lh s$, and
$u(\lh s + i) = t(i)$ for every $i <\lh t$. These notations are extended to
infinite sequences (i.e., functions with domain $\N$) as well.

If $X \subseteq \N$ is infinite we denote by $\sq X$ the set of all finite
subsets of $X$. We identify an element of $\sq \N$ with the unique element of
\Seq\ which enumerates it in increasing order, so that we can use the
notation introduced above. If $k \in \N$, $[X]^k$ is the subset of $\sq X$
consisting of the sets with exactly $k$ elements. Similarly $\Sq X$ stands
for the collection of all infinite subsets of $X$. Note that $\Sq X$ does not
formally exist in second order arithmetic, and is only used in expressions of
the form $Y \in \Sq X$; here again we identify $Y$ with the unique sequence
enumerating it in increasing order (notice that in \RCA\ an element of $\Sq
X$ exists as a set if and only if it exists as an increasing sequence, so
that this identification is harmless). For $X \in \Sq \N$ let $X^- = X
\setminus \{\min X\}$, i.e., $X$ with its least element removed. Similarly if
$s \in \sq \N$ is nonempty we set $s^- = s \setminus \{\min s\}$.

If $B \subseteq \sq \N$ then $\base (B)$ is the set
\[
\set{n}{\exists s \in B \, \exists i < \lh s \, s(i) = n}.
\]
\RCA\ does not prove the existence of $\base (B)$ for arbitrary $B \subseteq
\sq \N$; indeed in \cite[Lemma 1.4]{wqobqo} it is shown that, over \RCA,
\ACA\ is equivalent to the assertion that $\base (B)$ exists as a set for
every $B \subseteq \sq \N$. However this does not affect the possibility of
defining blocks and barriers within \RCA: e.g., ``$\base (B)$ is infinite''
(which is condition (1) in the definition of block below) can be expressed by
$\forall m\, \exists n > m\, \exists s \in B\, n \in s$. Similarly, when we
say $X$ is a subset of $\base (B)$ (for example in condition (2) of the
definition of block), we mean $\forall x \in X\, \exists s \in B\, x \in s$.
After giving the definitions, Lemma \ref{base} below will show that in fact
\RCA\ proves that $\base(B)$ exists whenever $B$ is a block (and, a fortiori,
a barrier).

\begin{definition}[\RCA]\label{def:barrier}
A set $B \subseteq \sq \N$ is a \emph{block} if:%
\begin{enumerate}[\quad (1)]
\item $\base(B)$ is infinite;
\item $\forall X \in \Sq{\base(B)}\, \exists s \in B\, s \init X$;
\item $\forall s,t \in B\, s \not\init t$.
\end{enumerate}
$B$ is a \emph{barrier} if it satisfies (1), (2) and
\begin{enumerate}[\quad (1')]\setcounter{enumi}2
\item $\forall s,t \in B\, s \not\subset t$.
\end{enumerate}
\end{definition}

Within \RCA\ it is immediate that every barrier is a block and we can check
that $[\N]^k$ (for $k>0$), $\set{s \in \sq \N}{\lh s = s(0) +1}$ and $\set{s
\in \sq \N}{\lh s = s(s(0)) +1}$ are barriers.

Notice that if $B$ is a block and $Y \in \Sq{\base(B)}$ then \RCA\ proves
that there exists a unique block $B' \subseteq B$ such that $\base(B') = Y$:
in fact $B' = \set{s \in B}{s \subset Y}$. Moreover if $B$ is a barrier then
$B'$ is also a barrier and we say that $B'$ is a \emph{subbarrier} of $B$.

The following result is Lemma 5.5 of \cite{defwqo}.

\begin{lemma}[\RCA]\label{base}
If $B$ is a block then $\base (B)$ exists as a set and $B$ is isomorphic to a
block $B'$ with $\base(B') = \N$.
\end{lemma}

\begin{definition}[\RCA]\label{tri}
Let $s,t \in \sq \N$: we write $s \tri t$ if there exists $u \in \sq \N$ such
that $s \initeq u$ and $t \initeq u^-$.
\end{definition}

Notice that $\langle 2,4,9 \rangle \tri \langle 4,9,10,14 \rangle \tri
\langle 9,10,14,21 \rangle$ and $\langle 2,4,9 \rangle \ntri \langle
9,10,14,21 \rangle$, so that $\tri$ is not transitive.

\begin{definition}[\RCA]\label{good/bad}
Let $(Q, {\preceq})$ be a quasi-order, $B$ be a block and $f: B \to Q$. We
say that $f$ is \emph{good} (with respect to $\preceq$) if there exist $s,t
\in B$ such that $s \tri t$ and $f(s) \preceq f(t)$. If $f$ is not good then
we say that it is \emph{bad}. $f$ is \emph{perfect} if for every $s,t \in B$
such that $s \tri t$ we have $f(s) \preceq f(t)$.
\end{definition}

We can now give the definition of bqo:

\begin{definition}[\RCA]\label{def:bqo}
Let $(Q, {\preceq})$ be a quasi-order.
\begin{itemize}
  \item $Q$ is \emph{bqo} if for every barrier $B$ every $f: B \to Q$ is
      good with respect to $\preceq$;
  \item $Q$ is \emph{bqo(block)} if for every block $B$ every $f: B \to Q$
      is good with respect to $\preceq$.
\end{itemize}\end{definition}

An alternative definition of bqo was given by Simpson in \cite{SinMW}. A
block $B$ represents an infinite partition of $\Sq{\base (B)}$ into clopen
sets with respect to the topology that $\Sq{\base (B)}$ inherits from
$\N^\N$. Thus any $f: B \to Q$ represents a continuous function $F: \Sq{\base
(B)} \to Q$ where $Q$ has the discrete topology; $f$ is good if for some $X
\in \Sq{\base (B)}$ we have $F (X) \preceq F (X^-)$. Therefore $(Q,
{\preceq})$ is bqo if and only if for every continuous function $F: \Sq{\base
(B)} \to Q$ there exists $X \in \Sq{\base (B)}$ such that $F (X) \preceq F
(X^-)$. Moreover if we replace continuous with Borel we are still defining
the same notion (this follows from the fact, originally proved by Mathias,
that for every Borel function $F: \Sq{\base (B)} \to Q$ there exists $X \in
\Sq{\base (B)}$ such that the restriction of $F$ to $\Sq X$ is continuous).
We are not discussing these alternative characterizations of bqo here, but
they have been exploited by Montalb\'{a}n in his proof of Theorem \ref{FC} below.

It is easy to see (using the barrier $[\N]^1$ and the fact that $\langle m
\rangle \tri \langle n \rangle$ if and only if $m<n$) that \RCA\ proves that
every bqo is wqo.

Lemma \ref{base} shows that within \RCA\ we can restrict the definition of
bqo and bqo(block) to functions with domain barriers or blocks with base \N.
It is also immediate that every bqo(block) is also a bqo. For the opposite
implication, we have the following result \cite[Theorem 5.12]{defwqo}.

\begin{lemma}[\WKL]\label{bqo->bqo(block)}
Every bqo is bqo(block).
\end{lemma}

The natural proof that every bqo is bqo(block) uses the clopen Ramsey
theorem, which is equivalent to \ATR, to show that every block contains a
barrier. The proof of Lemma \ref{bqo->bqo(block)} instead exploits a
construction originally appeared in \cite{fbqo} and builds a barrier which is
connected to, but in general not included in, the original block.

Lemma \ref{bqo->bqo(block)} leads to the following question:

\begin{question}
Is \lq\lq every bqo is bqo(block)\rq\rq\ equivalent to \WKL\ over \RCA?
\end{question}

Another characterization of bqos corresponds to the wqo(set) characterization
of wqos.

\begin{definition}
A quasi-order $(Q, {\preceq})$ is \emph{bqo(set)} if for every barrier $B$
and every $f: B \to Q$ there exists a subbarrier $B' \subseteq B$ such that
$f$ restricted to $B'$ is perfect with respect to $\preceq$.
\end{definition}

\RCA\ trivially proves that every bqo(set) is bqo, while the reverse
implication is known to be much stronger (see \cite[Theorem 4.9]{wqobqo},
which revisits \cite[Lemma V.9.5]{sosoa}).

\begin{theorem}[\RCA]
The following are equivalent:
\begin{enumerate}[\quad (i)]
  \item \ATR;
  \item every bqo is bqo(set).
\end{enumerate}\end{theorem}\smallskip

It is easy to realize that \RCA\ suffices to prove that every well-order is
bqo, and even bqo(block) (see \cite[Lemma 3.1]{wqobqo}). Dealing with finite
quasi-orders is however more problematic. Let $n$ denote the partial order
consisting of $n$ mutually incomparable elements, and notice that if $n$ is
bqo, or bqo(block), or bqo(set), then every quasi-order with the same number
of elements has the same property. The following results are from \cite[Lemma
3.2, Theorem 5.11 and Theorem 4.9]{wqobqo}.

\begin{theorem}\label{finitebqo}
\begin{enumerate}
  \item \RCA\ proves that $2$ is bqo and bqo(block);
  \item \ATR\ proves that $3$ is bqo;
  \item for any fixed $n \geq 3$, \RCA\ proves that $3$ is bqo is
      equivalent to $n$ is bqo;
\item for any fixed $n \geq 2$, \RCA\ proves that $n$ is bqo(set) is
    equivalent to \ATR.
\end{enumerate}
\end{theorem}

Item \rr{3} above leads to the following question, which was already stated
as Problem 3.3 in \cite{wqobqo}.

\begin{question}\label{q:3bqo}
What is the strength of the statement \lq\lq $3$ is bqo\rq\rq?
\end{question}

Over the years, the author has involved several colleagues in trying to
attack this problem, but no progress has been made. We devote some time to
explain the situation. The $\tri$ relation can be viewed as defining a graph
with the elements of $\Sq \N$ as vertices. The assertion that $n$ is bqo
amounts to state that the subgraph whose set of vertices is a barrier is not
$n$-colorable. Indeed, the proof of item \rr{1} of Theorem \ref{finitebqo}
amounts to the definition within \RCA\ of a cycle of odd length inside any
barrier or block. It is much more difficult to show that a graph is not
$3$-colorable, and this accounts for the increased difficulty in showing that
$3$ is bqo. A first step in beginning to answer Question \ref{q:3bqo} would
be showing that the $\omega$-model \REC\ does not satisfy that every barrier
is $3$-colorable. To this end one cannot use a computable barrier $B$: in
fact being $3$-colorable is an arithmetic property, and hence surely false
for $B$ in \REC. What is needed is some $B \subseteq \Sq \N$ which looks like
a barrier in \REC\ (i.e., which satisfies \rr{1} and \rr{3'} of Definition
\ref{def:barrier} and is such that for every computable $X \in \Sq{\base(B)}$
there exists $s \in B$ with $s \init X$), but is $3$-colorable.\smallskip

Moving now to the basic closure properties of bqos, we start by noticing the
following obvious fact, which mirrors the results of Lemma \ref{RCAsubset}
about wqos.

\begin{lemma}[\RCA]\label{RCAsubsetbqo}
Let $\mathcal{P}$ be any of the properties bqo, bqo(block) or bqo(set). If
$(Q,{\preceq})$ satisfies $\mathcal{P}$ and $R \subseteq Q$ then the
restriction of $\preceq$ to $R$ satisfies $\mathcal{P}$ as well.
\end{lemma}

Only part of Lemma \ref{+duwqo} has an analogous for bqos.

\begin{lemma}[\RCA]\label{+bqo}
Let $\mathcal{P}$ be any of the properties bqo, bqo(block) or bqo(set). If
$Q_1$ and $Q_2$ satisfy $\mathcal{P}$ then $Q_1 + Q_2$ satisfies
$\mathcal{P}$ with respect to the sum quasi-order.
\end{lemma}

When $\mathcal{P}$ is bqo this is \cite[Lemma 5.14]{wqobqo}. The proof shows
that for any $f: B \to Q_1 + Q_2$ there is a subbarrier $B'$ such that the
restriction of $f$ to $B'$ has range in $Q_i$ for some $i$: this yields the
result also when $\mathcal{P}$ is bqo(set). Moreover the proof works also for
blocks, thus taking care of the case when $\mathcal{P}$ is bqo(block).

The closure under disjoint unions of bqos is much stronger than the
corresponding property for wqos. In fact we have

\begin{lemma}[\RCA]\label{du=cup=times}
Let $\mathcal{P}$ be any of the properties bqo, bqo(block) or bqo(set). The
following are equivalent:
\begin{enumerate}[\quad (i)]
    \item if $Q_1$ and $Q_2$ satisfy $\mathcal{P}$ then $Q_1 \du Q_2$
        satisfies $\mathcal{P}$ with respect to the disjoint union
        quasi-order;
    \item if $Q$ satisfies $\mathcal{P}$ with respect to the quasi-orders
        $\preceq_1$ and $\preceq_2$ then $Q$ satisfies $\mathcal{P}$ with
        respect to the intersection quasi-order;
    \item if $Q_1$ and $Q_2$ satisfy $\mathcal{P}$ then $Q_1 \times Q_2$
        satisfies $\mathcal{P}$ with respect to the product quasi-order.
\end{enumerate}
All these statements are provable in \ATR. When $\mathcal{P}$ is bqo or
bqo(block) they imply \ACA, when $\mathcal{P}$ is bqo(set) they are
equivalent to \ATR.
\end{lemma}

The equivalence between the three statements for bqo is Lemma 5.16 of
\cite{wqobqo}: the implication from \rr{i} to \rr{iii} uses Theorem
\ref{RCAbqoflatsharp} below. The same proof works also for bqo(block) and
bqo(set). Provability in \ATR\ follows easily from the clopen Ramsey theorem.
The implication towards \ACA\ is Lemma 5.17 of \cite{wqobqo} (which uses the
proof of Theorem \ref{ACAHig} below) when we are dealing with bqos, and works
also for bqo(block). The implication towards \ATR\ is immediate from item
\rr{4} of Theorem \ref{finitebqo} because \rr{i} for bqo(set) implies that
$2$ is bqo(set).

\begin{question}\label{q:closurebqodu}
What is the strength of statements \rr i--\rr{iii} of Lemma
\ref{du=cup=times} when $\mathcal{P}$ is bqo or bqo(block)?
\end{question}

Since the statements imply \ACA, by Lemma \ref{bqo->bqo(block)} there is a
single answer for bqo and bqo(block). Since \rr i for bqo implies that $3$ is
bqo, Questions \ref{q:closurebqodu} and \ref{q:3bqo} are connected.

\section{Minimality arguments}\label{sec:mb}
One of the main tools of wqo theory is the minimal bad sequence lemma
(apparently isolated for the first time in \cite{NW63}). The idea is to prove
that a quasi-order is wqo by showing that if there exists a bad sequence then
there is one with a minimality property, and eventually reaching a
contradiction from the latter assumption. To state the lemma in its general
form we need the following definitions.

\begin{definition}[\RCA]\label{compatible}
Let $(Q,{\preceq})$ be a quasi-order. A transitive binary relation $<'$ on
$Q$ is \emph{compatible with $\preceq$} if for every $x,y \in Q$ we have that
$x <' y$ implies $x \preceq y$. We write $x \leq' y$ for $x <' y \lor x = y$.
In this situation, if $A, A' \in \Sq \N$, $f: A \to Q$, and $f': A' \to Q$ we
write $f \leq' f'$ if $A \subseteq A'$ and $\forall n \in A\; f(n) \leq'
f'(n)$; we write $f <' f'$ if $f \leq' f'$ and $\exists n \in A\, f(n) <'
f'(n)$. $f$ is \emph{minimal bad with respect to $<'$} if it is bad with
respect to $\preceq$ and there is no $f' <' f$ which is bad with respect to
$\preceq$.
\end{definition}

\begin{stat}[minimal bad sequence lemma]
Let $(Q,{\preceq})$ be a quasi-order and $<'$ a well-founded relation which
is compatible with $\preceq$: if $A' \in \Sq \N$ and $f': A' \to Q$ is bad
with respect to $\preceq$ then there exists $f: A \to Q$ such that $f \leq'
f'$ and $f$ is minimal bad with respect to $<'$.
\end{stat}

The generalization of the minimal bad sequence lemma to bqos is known as the
minimal bad array lemma (the maps of definition \ref{good/bad} are sometimes
called arrays) or the forerunning technique (this method was explicitly
isolated and clarified in \cite{Laver78}). Again, we need some preliminary
definitions.

\begin{definition}[\RCA]
Let $(Q,{\preceq})$ be a quasi-order and $<'$ be compatible with $\preceq$ in
the sense of definition \ref{compatible}. If $B$ and $B'$ are barriers, $f: B
\to Q$, and $f': B' \to Q$ we write $f \leq' f'$ if $\base (B) \subseteq
\base (B')$, and for every $s \in B$ there exists $s' \in B'$ such that $s'
\initeq s$ and $f(s) \leq' f'(s')$. We write $f <' f'$ if $f \leq' f'$ and
for some $s \in B$, $s' \in B'$ with $s' \initeq s$ we have $f(s) <' f'(s')$.
$f$ is \emph{minimal bad with respect to $<'$} if it is bad with respect to
$\preceq$ and there is no $f' <' f$ which is bad with respect to $\preceq$.
\end{definition}

\begin{stat}[minimal bad array lemma]
Let $(Q,{\preceq})$ be a quasi-order and $<'$ a well-founded relation which
is compatible with $\preceq$. If $B'$ is a barrier and $f': B' \to Q$ is bad
with respect to $\preceq$ then there exist a barrier $B$ and $f: B \to Q$
such that $f \leq' f'$ and $f$ is minimal bad with respect to $<'$.
\end{stat}

A milder generalization of the minimal bad sequence lemma is also useful: it
was actually the first version of the minimal bad array lemma proved for a
specific quasi-order by Nash-Williams in \cite{NW65b} and was isolated in
\cite{fbqo}.

\begin{definition}[\RCA]
Let $(Q,{\preceq})$ be a quasi-order and $<'$ be compatible with $\preceq$ in
the sense of definition \ref{compatible}. If $B$ and $B'$ are barriers, $f: B
\to Q$, and $f': B' \to Q$ we write $f \leq'_\ell f'$ if $B \subseteq B'$ and
$\forall s \in B\; f(s) \leq' f'(s)$. We write $f <'_\ell f'$ if $f
\leq'_\ell f'$ and $\exists s \in B\; f(s) <' f'(s)$. $f$ is \emph{locally
minimal bad with respect to $<'$} if it is bad with respect to $\preceq$ and
there is no $f' <'_\ell f$ which is bad with respect to $\preceq$.
\end{definition}

\begin{stat}[locally minimal bad array lemma]
Let $(Q,{\preceq})$ be a quasi-order and $<'$ a well-founded relation which
is compatible with $\preceq$: if $B'$ is a barrier and $f': B' \to Q$ is bad
with respect to $\preceq$ then there exist a barrier $B$ and $f: B \to Q$
such that $f \leq'_\ell f'$ and $f$ is locally minimal bad with respect to
$<'$.
\end{stat}

The minimal bad sequence lemma and the locally minimal bad array lemma have
been shown to be equivalent to the strongest of the big five by Simpson and
Marcone in \cite[Theorem 6.5]{NWT}.

\begin{theorem}[\RCA]\label{mbs}
The following are equivalent:
\begin{enumerate}[\quad (i)]
  \item \PCA;
  \item the minimal bad sequence lemma;
  \item the locally minimal bad array lemma.
\end{enumerate}
\end{theorem}

On the other hand, the proofs of the minimal bad array lemma use very strong
set-existence axioms: a crude analysis shows that they can be carried out
within \PPCA.

\begin{question}
What is the axiomatic strength of the minimal bad array lemma?
\end{question}

\section{Structural results}\label{sec:struct}
In this section we consider theorems showing that wqos satisfy specific
properties as partial orders.

The better known of these theorems is due to de Jongh and Parikh \cite{JP}
(an exposition of essentially the original proof appears in
\cite[\S8.4]{Harz}; a proof based on the study of the partial order of the
initial segments of the wqo is included in \cite[\S4.11]{Fra00}; proofs with
a strong set-theoretic flavor appear as \cite[Theorem 4.7]{KT} and
\cite[Proposition 52]{BG}).

\begin{stat}[maximal linear extension theorem]
If $(Q,{\preceq})$ is wqo, then there exist a linear extension $\preceq_L$ of
$Q$ which is maximal, meaning that every linear extension of $Q$ embeds in an
order-preserving way into $\preceq_L$.
\end{stat}

A less known result is due to Wolk (\cite[Theorem 9]{Wolk}, actually Wolk's
statement is slightly stronger) and also appears as \cite[Theorem 4.9]{KT}
and \cite[Theorem 8.1.7]{Harz}.

\begin{stat}[maximal chain theorem]
If $(Q,{\preceq})$ is wqo, then there exist a chain $C \subseteq Q$ which is
maximal, meaning that every chain contained in $Q$ embeds in an
order-preserving way into $C$.
\end{stat}

Marcone and Shore \cite{MS11} studied the strength of the maximal linear
extension theorem and of the maximal chain theorem.

\begin{theorem}[\RCA]\label{MLE}
The following are equivalent:
\begin{enumerate}[\quad (i)]
  \item \ATR;
  \item the maximal linear extension theorem;
  \item the maximal chain theorem.
\end{enumerate}
\end{theorem}

The proofs of the two theorems within \ATR\ differ from the proofs found in
the literature: to avoid using more induction than available in \ATR\ one
fixes a wqo $Q$ and looks respectively at the \emph{tree of finite bad
sequences in $Q$}
\[
\Bad(Q) = \set{s \in \Qom}{\forall i,j < \lh s (i<j \rightarrow s(i) \npreceq s(j))}
\]
and at the \emph{tree of descending sequences in $Q$}
\[
\Desc(Q) = \set{s \in \Qom}{\forall i,j < \lh s (i<j \rightarrow s(j) \prec s(i))}.
\]
(Here $\Qom$ is the set of finite sequences of elements of $Q$.) Since $Q$ is
wqo both these trees are well-founded and \ATR\ can compute their rank
functions. Focusing on the maximal linear extension theorem (the other proof
follows the same strategy), by recursion on the rank of $s \in \Bad(Q)$ we
assign to $s$ a maximal linear extension of the restriction of $\preceq$ to
$\set{x \in Q}{s \conc \langle x \rangle \in \Bad(Q)}$; when $s$ is the empty
sequence we have the maximal linear extension of $Q$.

The two reversals contained in Theorem \ref{MLE} have quite different proofs.
The proof that the maximal chain theorem implies \ATR\ is very simple (using
the well-known equivalence between \ATR\ and comparability of well-orders),
while the proof that the maximal linear extension theorem implies \ATR\ is
more involved. In fact there is first a bootstrapping, showing that the
maximal linear extension theorem implies \ACA. To this end it is useful a
partial order $Q$ such that the existence of any bad sequence in $Q$ implies
\ACA: thus if \ACA\ fails then $Q$ is wqo, we can apply the theorem and reach
a contradiction from the existence of a maximal linear extension. We can now
argue within \ACA\ and, assuming the failure of \ATR\ and using Theorem
\ref{Shore} below, build a wqo $Q'$ which cannot have a maximal linear
extension. The difference of the two proofs is no accident. In fact a theorem
of Montalb\'{a}n \cite{Mont07} states that every computable wqo has a computable
maximal linear extension (this implies that in showing that the maximal
linear extension theorem implies anything unprovable in \RCA\ the use of
partial orders that are not really wqos is unavoidable), while Marcone,
Montalb\'{a}n and Shore \cite[Theorem 3.3]{MMS} showed that for every
hyperarithmetic set $X$ there is a computable wqo $Q$ with no $X$-computable
maximal chain.\smallskip

Another kind of structural theorems about quasi-orders concerns the
decomposability of the quasi-order in finite pieces which are simple.

\begin{definition}[\RCA]
Let $(Q,{\preceq})$ be a quasi-order. $I \subseteq Q$ is an \emph{ideal} if
\begin{itemize}
  \item $\forall x,y \in Q (x \in I \land y \preceq x \rightarrow y \in
      I)$;
  \item $\forall x,y \in I\, \exists z \in I (x \preceq z \land y \preceq
      z)$.
\end{itemize}
\end{definition}

Bonnet \cite[Lemma 2]{Bon73} (see also \cite[\S4.7.2]{Fra00}) proved that a
partial order has no infinite antichains if and only if every initial
interval is a finite union of ideals (this result follows also from
\cite[Theorem 1]{ErdTar43}). In \cite[Theorem 4.5]{init} Frittaion and
Marcone studied the left to right direction of Bonnet's result and proved,
among other things, the following equivalence.

\begin{theorem}[\RCA]
The following are equivalent:
\begin{enumerate}[\quad (i)]
  \item \ACA;
  \item every wqo is a finite union of ideals.
\end{enumerate}
\end{theorem}

\section{Major theorems about wqos and bqos}\label{sec:theorems}

In this section we consider the major theorems about wqos and bqos, starting
with Higman's basic result, first proved in \cite{Higman} and then
rediscovered many times.

\begin{definition}[\RCA]\label{Higman}
If $(Q,{\preceq})$ is a quasi-order we define a quasi-order on $\Qom$ by
setting $s \preceq^* t$ if and only if there exists an embedding of $s$ into
$t$, i.e., a strictly increasing $f: \lh s \to \lh t$ such that $s(i) \preceq
t(f(i))$ for every $i < \lh s$ (here $\lh s$ is the length of the sequence
$s$).
\end{definition}

\begin{stat}[Higman's theorem]
If $Q$ is wqo then $(\Qom,{\preceq^*})$ is wqo.
\end{stat}

Before analyzing Higman's theorem from the reverse mathematics viewpoint, let
us introduce other constructions of new quasi-orders starting from the one on
$Q$.

We denote by $\Pow X$ and $\Pf X$ the powerset of $X$ and the set of all
finite subsets of $X$. If $X$ is infinite $\Pow X$ does not exists as a set
in second order arithmetic, but we can define and study relations between
elements of $\Pow X$. A \emph{quasi-order on $\Pow X$} is just a formula
$\varphi$ with two distinguished set variables such that $\varphi(Y,Y)$ holds
and $\varphi(Y,Z)$ and $\varphi(Z,W)$ imply $\varphi(Y,W)$ whenever $Y,Z,W
\subseteq X$. We use symbols like $\preceq$ and infix notation to denote
quasi-orders on $\Pow X$.

\begin{definition}[\RCA]\label{uncountable}
If $\preceq$ is a quasi-order on $\Pow X$, a sequence $(X_n)_{n \in \N}$ of
elements of $\Pow X$ is \emph{good} (with respect to $\preceq$) if there
exist $m <_\N n$ such that $X_m \preceq X_n$. If every such sequence is good
we say that $\preceq$ is \emph{wqo}.

Analogously, a sequence $(X_s)_{s \in B}$ of elements of $\Pow X$ indexed by
a barrier $B$ is \emph{good} (with respect to $\preceq$) if there exist $s,t
\in B$ such that $s \tri t$ and $X_s \preceq X_t$. If every such sequence is
good we say that $\preceq$ is \emph{bqo}.
\end{definition}

The following two quasi-orders are called the \emph{Hoare quasi-order} and
the \emph{Smyth quasi-order} in the computer science literature. (Here we
follow the computer science notation: in \cite{wqobqo} $\preceq^\flat$ was
written as $\preceq^\exists_\forall$ and $\preceq^\sharp$ as
$\preceq^\forall_\exists$.)

\begin{definition}[\RCA]
Let $(Q,{\preceq})$ be a quasi-order. If $X, Y \in \Pow Q$ let
\begin{align*}
  X \preceq^\flat Y & \iff \forall x \in X\, \exists y \in Y\, x \preceq y \quad
        \text{ and}\\
  X \preceq^\sharp Y & \iff \forall y \in Y\, \exists x \in X\, x \preceq y.
\end{align*}
\end{definition}

\begin{theorem}[\RCA]\label{ACAHig}
The following are equivalent:
\begin{enumerate}[\quad (i)]
  \item \ACA;
  \item Higman's theorem;
  \item if $Q$ is wqo then $(\Pf Q, {\preceq^\flat})$ is wqo.
\end{enumerate}\end{theorem}

Most proofs of Higman's theorem are based on the minimal bad sequence lemma.
Theorem \ref{mbs} implies that such a proof cannot be carried out in \ACA. In
fact, the provability of Higman's theorem in \ACA\ is based on the technique
of reification of wqos by well-orders (\cite{JP,Schmidt}, see also \cite{KT})
and follows from the results in Section 4 of \cite{Simpson88b} (see
\cite[Theorem 3]{Clote90} for details). A \emph{reification} of $Q$ by the
linear order $(X,{\leq_X})$ is a map $\rho$ from $\Bad(Q)$ to $X$ such that
$\rho(t) <_X \rho(s)$ whenever $s \init t$. Thus, if $X$ is a well-order then
$\rho$ is an approximation to the rank function on $B(Q)$, and suffices to
witness that $\Bad(T)$ is well-founded and hence $Q$ is wqo.

\ACA\ is used twice in this proof: first to show that every wqo admits a
reification by a well-order and then to show that
$\omega^{\omega^{\alpha+1}}$ is a well-order when $\alpha$ is a well-order
(closure of well-orders under exponentiation is equivalent to \ACA\ over
\RCA\ by \cite{Girard}, see \cite{Hirst94}). This suffices, because \RCA\
proves that if $Q$ has a reification of order type $\alpha$ then $\Qom$ has a
reification of order type $\omega^{\omega^{\alpha+1}}$ (\cite[Sublemma
4.8]{Simpson88b}, which is Lemma 5.2 of \cite{Schutte-Simpson}) and that if a
quasi-order admits a reification by a well-order then it is wqo.

\RCA\ clearly suffices to prove that \rr{ii} implies \rr{iii} of Theorem
\ref{ACAHig}, while the implication from \rr{iii} to \ACA\ was proved in
\cite{wqobqo} using \RT22 (but in fact only closure of wqo under product was
necessary); the provability of this implication in \RCA\ was shown in
\cite[Theorem 2.5]{Noeth}.

If $Q$ is wqo then in general neither $\Pow Q$ with respect to
$\preceq^\flat$ nor $\Pf Q$ with respect to $\preceq^\sharp$ are wqo. However
if we strengthen the hypothesis to $Q$ bqo we obtain some true statements
which have been studied from the reverse mathematics viewpoint. The following
theorems summarize Theorems 5.4 and 5.6 in \cite{wqobqo}.

\begin{theorem}[\RCA]\label{RCAbqoflatsharp}
If $Q$ is bqo then $(\Pf Q, {\preceq^\flat})$ and $(\Pow Q,
{\preceq^\sharp})$ (and hence, a fortiori, also $(\Pf Q, {\preceq^\sharp})$)
are bqo.
\end{theorem}

\begin{theorem}[\ACA]\label{ACAbqoflat}
If $Q$ is bqo then $(\Pow Q, {\preceq^\flat})$ is bqo.
\end{theorem}

\begin{question}\label{q:ACAbqoflat}
Is the statements \lq\lq if $Q$ is bqo then $(\Pow Q, {\preceq^\flat})$ is
bqo\rq\rq\ equivalent to \ACA\ over \RCA?
\end{question}

Trying to answer affirmatively the previous question, one is faced with the
problem of applying the statement to a quasi-order $Q$ which is proved to be
bqo within \RCA. Such a $Q$ must be infinite (otherwise $\Pow Q = \Pf Q$ and
Theorem \ref{RCAbqoflatsharp} applies) and, unless the answer to Question
\ref{q:3bqo} is \RCA, with antichains of size at most $2$.

More results about the Hoare and Smyth quasi-orders (obtained by weakening
the conclusion) will be discussed in Section \ref{sec:topo} below.\smallskip

Another important result about wqos is Kruskal's theorem \cite{Kruskal60},
establishing a conjecture of V\'{a}zsonyi from the 1930's popularized by
Erd\H{o}s. This theorem deals with trees viewed as partial orders: for our
purposes we can represent them in second-order arithmetic as subsets of \Seq\
closed under initial segments.

\begin{definition}[\RCA]
Let $\mathcal{T}$ be the set of all finite trees. If $T_0, T_1 \in
\mathcal{T}$ let $T_0 \preceq_\mathcal{T} T_1$ if and only if there exists a
homeomorphic embedding of $T_0$ in $T_1$, that is, an injective $f: T_0 \to
T_1$ such that $f (s \land t) = f(s) \land f(t)$ for every $s,t \in T_0$
(where $s \land t$ is the greatest lower bound of $s$ and $t$, which is the
longest common initial segment of the two sequences).

If $Q$ is a set let $\mathcal{T}^Q$ be the set of finite trees labelled with
elements of $Q$, that is, pairs $(T, \ell)$ such that $T \in \mathcal{T}$ and
$\ell$ is a function from $T$ to $Q$.

If $(Q,{\preceq})$ is a quasi-order and $(T_0, \ell_0), (T_1, \ell_1) \in
\mathcal{T}^Q$ let $(T_0, \ell_0) \preceq_{\mathcal{T}^Q} (T_1, \ell_1)$ if
and only if there exists a homeomorphic embedding $f$ of $T_0$ in $T_1$ such
that $\ell_0 (s) \preceq \ell_1 (f(s))$ for every $s \in T_0$.
\end{definition}

\RCA\ easily shows that $\preceq_\mathcal{T}$ and $\preceq_{\mathcal{T}^Q}$
are quasi-orders.

\begin{stat}[Kruskal's theorem]
If $Q$ is wqo then $(\mathcal{T}^Q, {\preceq_{\mathcal{T}^Q}})$ is wqo.
\end{stat}

The usual proof of Kruskal's theorem uses the minimal bad sequence lemma and
can be carried out in \PCA\ using Theorem \ref{mbs}. On the other hand, this
statement is \PI12 and hence cannot imply \PCA\ (see \cite[Corollary
1.10]{NWT}).

Harvey Friedman proved the following striking result (see \cite{Simpson85b}).

\begin{theorem}\label{KT}
\ATR\ does not prove that $(\mathcal{T}, {\preceq_{\mathcal{T}}})$ is wqo. A
fortiori Kruskal's theorem is not provable in \ATR.
\end{theorem}

To prove this theorem we build a map $\psi$ between $\mathcal{T}$ and a
certain primitive recursive notation system for the ordinals less than
$\Gamma_0$, and show that \ACA\ proves that $\psi (T_0) \leq_o \psi (T_1)$
(where $\leq_o$ is the order on the ordinal notation system) whenever $T_0
\preceq_\mathcal{T} T_1$. Thus \ACA\ proves that if $(\mathcal{T},
{\preceq_{\mathcal{T}}})$ is wqo then the system of ordinal notations is a
well-order. Since $\Gamma_0$ is the proof-theoretic ordinal of \ATR, it
follows that \ATR\ does not prove that $(\mathcal{T},
{\preceq_{\mathcal{T}}})$ is wqo.

A lower bound for Kruskal's theorem is provided by the following theorem,
that apparently has never been explicitly stated.

\begin{theorem}[\RCA]\label{KT->ATR}
Kruskal's theorem implies \ATR.
\end{theorem}
\begin{proof}[Sketch of proof]
We use the fact that \ATR\ is equivalent, over \RCA, to the statement that if
$X$ is a well-order then $\boldsymbol \varphi(X,0)$ is a well-order, where
$\boldsymbol \varphi$ is the formalization of the Veblen function on the
ordinals. This theorem was originally proved by H. Friedman (unpublished) and
then given a proof-theoretic proof in \cite{RW11} and a
computability-theoretic proof in \cite{Veblen}. We follow the notation of the
latter paper.

To prove our theorem first notice that Kruskal's theorem generalizes Higman's
theorem, so that we can argue in \ACA. Given a well-order $X$ we can mimic
the construction of the proof of Theorem \ref{KT} using $X$ as the set of
labels for the finite trees. In this way we define a map $\psi$ between
$\mathcal{T}^X$ and the ordinals less than the first fixed point for the
Veblen function strictly larger than $X$. We then show that \ACA\ proves that
$\psi (T_0) \leq_\varphi \psi (T_1)$ whenever $T_0 \preceq_{\mathcal{T}^X}
T_1$. Since our hypothesis implies that $(\mathcal{T}^X,
{\preceq_{\mathcal{T}^X}})$ is wqo we obtain that $(\boldsymbol \varphi(X,0),
{\leq_\varphi})$ is a well-order, as needed.
\end{proof}

Thus Kruskal's theorem is properly stronger than \ATR\ and provable in, but
not equivalent to, \PCA. In an attempt to classify statements of this kind,
Henry Towsner \cite{Tow:imp} introduced a sequence of intermediate systems
based on weakening the leftmost path principle (which is equivalent to \PCA).
Towsner tested his approach by looking at various statements and, by
analyzing Nash-Williams' proof of Kruskal's theorem, obtained the following
result.

\begin{theorem}\label{Tow:KT}
Kruskal's theorem is provable in Towsner's system
$\Sigma_2$-$\mathsf{LPP}_0$.
\end{theorem}

Unfortunately no reversal to Towsner's systems are known, so we do not know
whether the upper bound for the strength of Kruskal's theorem provided by the
previous theorem is optimal.

Rathjen and Weiermann \cite{RW93} carried out a detailed proof-theoretic
analysis of the statement \lq\lq $(\mathcal{T}, {\preceq_{\mathcal{T}}})$ is
wqo\rq\rq\ (beware that Rathjen and Weiermann call Kruskal's theorem this
statement) showing that it is equivalent over \ACA\ to the uniform \PI11
reflection principle of the theory obtained by adding transfinite induction
for \PI12 formulas to \ACA.

Harvey Friedman, inspired by ordinal notation systems, introduced a
refinement of $\preceq_{\mathcal{T}}$ (obtained by requiring that the
homeomorphic embedding satisfies a \lq\lq gap condition\rq\rq) and proved
that it still yields a wqo on $\mathcal{T}$. Friedman himself
\cite{Simpson85b} showed, generalizing the technique of theorem \ref{KT} to
larger ordinals, that this wqo statement is not provable in \PCA.\smallskip

The most striking instance of this unprovability phenomenon is provided by
the graph minor theorem, proved by Robertson and Seymour in a long series of
papers (see \cite[Section 5]{Thomassen} or \cite[Chapter 12]{Diestel} for
overviews).

\begin{definition}[\RCA]
If $\mathcal{G}$ is the set of all finite directed graphs (allowing loops and
multiple edges) define a quasi-order on $\mathcal{G}$ by setting $G_0
\preceq_m G_1$ if and only if $G_0$ is isomorphic to a minor of $G_1$ (recall
that a minor is obtained by deleting edges and vertices and contracting
edges).
\end{definition}

\begin{stat}[graph minor theorem]
$\preceq_m$ is wqo on $\mathcal{G}$.
\end{stat}

Friedman's generalization of Kruskal's theorem mentioned above plays a
significant role in some steps of the proof of the graph minor theorem, which
uses iterated applications of the minimal bad sequence lemma. This proof
cannot be carried out in \PCA\ and the following theorem (proved by Friedman,
Robertson and Seymour \cite{FRS} well before the completion of the proof of
the graph minor theorem) shows that there is no simpler proof.

\begin{theorem}\label{gmt}
The graph minor theorem (and even special cases where $\preceq_m$ is
restricted to some subset of $\mathcal{G}$) is not provable in \PCA.
\end{theorem}

This theorem is proved once more generalizing the technique of theorem
\ref{KT} to larger ordinals. Notice also that the graph minor theorem is a
\PI11 statement, and therefore does not imply any set-existence axiom (in
fact it holds in every $\omega$-model). More recently Rathjen and Krombholz
\cite{KrombholzPhD,KrombholzRathjen} analyzed more in detail the proof by
Robertson and Seymour in search of upper bounds for the proof-theoretic
strength of this statement, showing that it can be carried out in the system
obtained by adding transfinite induction for \PI12 formulas to
\PCA.\smallskip

It is well-known that Higman's theorem does not extend to infinite sequences,
and the canonical counterexample is Rado's partial order. The notion of bqo
was developed by Nash-Williams as a way of ruling out Rado's example and its
generalizations. Indeed, one of the first theorems of the subject is a
generalization of Higman's theorem \cite{NW68}.

\begin{definition}[\RCA]
If $(Q,{\preceq})$ is a quasi-order we can extend the quasi-order $\preceq^*$
of Definition \ref{Higman} from $\Qom$ to $\Qt$, the set of all countable
sequences of elements of $Q$ (i.e., the set of all functions from a countable
well-order to $Q$).
\end{definition}

\begin{stat}[Nash-Williams' theorem]
If $Q$ is bqo then $(\Qt,{\preceq^*})$ is bqo.
\end{stat}

Notice that $\Qt$ is uncountable, and hence we express \lq\lq
$(\Qt,{\preceq^*})$ is bqo\rq\rq\ in a way similar to Definition
\ref{uncountable}.

The following theorem is \cite[Theorem 4.5]{NWT}.

\begin{theorem}\label{NWT}
\PCA\ proves Nash-Williams' theorem.
\end{theorem}

The most natural proof of Nash-Williams' theorem uses the minimal bad array
lemma, and therefore to prove Theorem \ref{NWT} a new argument is needed.
This is obtained by using the locally minimal bad array lemma (provable in
\PCA\ by Theorem \ref{mbs}) to establish the following weak version of
Nash-Williams' theorem.

\begin{stat}[generalized Higman's theorem]
If $Q$ is bqo then $(\Qom,{\preceq^*})$ is bqo.
\end{stat}

Assuming the generalized Higman's theorem, we can prove Nash-Williams'
theorem in \ATR. Thus the proof of Theorem \ref{NWT} yields the following
result.

\begin{theorem}[\ATR]\label{GHT}
The following are equivalent:
\begin{enumerate}[\quad (i)]
  \item Nash-Williams' theorem;
  \item the generalized Higman's theorem.
\end{enumerate}\end{theorem}

Nash-Williams' theorem cannot imply \PCA, even over \ATR\ \cite[Theorem
5.7]{NWT}. In fact, the proof of Theorem \ref{NWT} actually establishes a
\PI12 statement that, over \ATR, implies Nash-Williams' theorem. The argument
mentioned before Theorem \ref{KT} then establishes the assertion. (Both
Nash-Williams' theorem and the generalized Higman's theorem are \PI13
statements, so we cannot apply the argument directly.) Towsner \cite{Tow:imp}
looked also at the proof of the locally minimal bad array lemma.

\begin{theorem}\label{Tow:NWT}
The generalized Higman's theorem, and therefore also Nash-Williams'
theorem, is provable in Towsner's system $\mathsf{TLPP}_0$.
\end{theorem}

$\mathsf{TLPP}_0$ is much stronger than the system
$\Sigma_2$-$\mathsf{LPP}_0$ appearing in Theorem \ref{Tow:KT}. Unfortunately,
as already mentioned, no reversal to Towsner's systems are known, so Theorem
\ref{Tow:NWT} provides just an upper bound for the strength of Nash-Williams'
theorem. Regarding lower bounds, Shore \cite{Shore} proved the following
important result.

\begin{theorem}[\RCA]\label{Shore}
The following are equivalent:
\begin{enumerate}[\quad (i)]
  \item \ATR
  \item every infinite sequence of countable well-orders contains two
      distinct elements which are comparable with respect to embeddability
      (as defined in Definition \ref{def:embed} below).
\end{enumerate}\end{theorem}

It is immediate that Nash-Williams' theorem implies \rr{ii}, and hence \ATR,
within \RCA.

\begin{question}\label{q:NWT}
Is Nash-Williams' theorem equivalent to \ATR?
\end{question}

A positive answer to this question was conjectured in
\cite{NWT,wqobqo}.\smallskip

Connected to Nash-Williams' theorem is one of the most famous achievements of
bqo theory, Laver's proof \cite{Laver71} of Fra\"{\i}ss\'{e}'s conjecture
\cite{Fraisse}. Laver actually proved a stronger result (even stronger than
the one we state below) and we keep the two statements distinct.

\begin{definition}[\RCA]\label{def:embed}
If $\mathcal{L}$ is the set of countable linear orderings define the
quasi-order of embeddability on $\mathcal{L}$ by setting $L_0
\preceq_\mathcal{L} L_1$ if and only if there exists an order-preserving
embedding of $L_0$ in $L_1$, i.e., an injective $f: L_0 \to L_1$ such that $x
<_{L_0} y$ implies $f (x) <_{L_1} f(y)$ for every $x,y \in L_0$.
%
%If $(Q,{\preceq})$ is a quasi-order define a quasi-order on $\mathcal{L}^Q$,
%the set of countable linear orderings labelled with elements of $Q$, by
%setting $(L_0, \ell_0) \preceq_{\mathcal{L}^Q} (L_1, \ell_1)$ if and only if
%there exists an order-preserving embedding $f$ of $L_0$ in $L_1$ such that
%$\ell_0 (x) \preceq \ell_1 (f(x))$ for every $x \in L_0$.
\end{definition}

\begin{stat}[Fra\"{\i}ss\'{e}'s conjecture]
$(\mathcal{L},{\preceq_\mathcal{L}})$ is wqo.
\end{stat}

\begin{stat}[Laver's theorem]
$(\mathcal{L}, {\preceq_{\mathcal{L}}})$ is bqo.
\end{stat}

Again, $\mathcal{L}$ is uncountable, and hence we express \lq\lq
$(\mathcal{L},{\preceq_\mathcal{L}})$ is wqo (bqo)\rq\rq\ by imitating
Definition \ref{uncountable}.

The strength of Fra\"{\i}ss\'{e}'s conjecture is one of the most important open
problems about the reverse mathematics of wqo and bqo theory. All known
proofs of Fra\"{\i}ss\'{e}'s conjecture actually establish Laver's theorem. Basically
only one proof was known until 2016: this proof uses the minimal bad array
lemma and can be carried out in \PPCA. Recently Antonio Montalb\'{a}n
\cite{FraPI11} made a major breakthrough by finding a new proof, which avoids
any form of \lq\lq minimal bad\rq\rq\ arguments. This proof is based on
Montalb\'{a}n's earlier analysis of Fra\"{\i}ss\'{e}'s conjecture \cite{MonFC} and uses
the Ramsey property for subsets of $\Sq \N$ and determinacy, yielding the
following result.

\begin{theorem}\label{FC}
\PCA\ proves Fra\"{\i}ss\'{e}'s conjecture and Laver's theorem.
\end{theorem}

Montalb\'{a}n defines \DE02-bqo by using \DE02 functions in Simpson's definition
of bqo given after Definition \ref{def:bqo}. Using the fact that \SI02 sets
are Ramsey (which is known to be equivalent to \PCA), he then shows that this
notion is equivalent to bqo. Within \ATR, using \DE01-determinacy (which is
equivalent to \ATR) Montalb\'{a}n proves that if $3$ is \DE02-bqo then Laver's
theorem holds. Since $3$ is bqo is provable in \ATR\ by item \rr{2} of
Theorem \ref{finitebqo} the proof of Theorem \ref{FC} in \PCA\ is then
complete.

Theorem \ref{Shore} entails that Fra\"{\i}ss\'{e}'s conjecture (and a fortiori Laver's
theorem) implies \ATR. Moreover Fra\"{\i}ss\'{e}'s conjecture is a \PI12 statement and
the usual considerations yield that \ATR\ plus Fra\"{\i}ss\'{e}'s conjecture cannot
imply \PCA. Montalb\'{a}n's proof shows that to prove Fra\"{\i}ss\'{e}'s conjecture in any
theory weaker than \PCA\ it suffices to prove that $3$ is \DE02-bqo. Thus an
unexpected connection with Question \ref{q:3bqo} comes up. Indeed, Montalb\'{a}n
shows that by mimicking the proof of item \rr1 of Theorem \ref{finitebqo} it
is easy to see that \ATR\ proves that $2$ is \DE02-bqo.

\begin{question}\label{q:FC}
Is Fra\"{\i}ss\'{e}'s conjecture equivalent to \ATR? Is \lq\lq $3$ is \DE02-bqo\rq\rq\
provable in \ATR?
\end{question}

A couple more results about Fra\"{\i}ss\'{e}'s conjecture are worth mentioning. First,
Montalb\'{a}n \cite{MonFC} showed that Fra\"{\i}ss\'{e}'s conjecture is equivalent, over
\RCA\ plus \SI11-induction, to a result about countable linear orders known
as Jullien's theorem. Therefore if the answer to Question \ref{q:FC} is
negative then Fra\"{\i}ss\'{e}'s conjecture defines a system intermediate between
\ATR\ and \PCA\ which is equivalent to other mathematical theorems.

On the other hand, Marcone and Montalb\'{a}n \cite{MM09} studied the restriction
of Fra\"{\i}ss\'{e}'s conjecture to linear orders of finite Hausdorff rank. To state
the result recall that \ACApl\ and \ACApr\ are obtained by adding to \RCA\
respectively \lq\lq for every $X$, $X^{(\omega)}$ (the arithmetic jump of
$X$) exists\rq\rq\ and \lq\lq for every $X$ and $k$, $X^{(k)}$ exists\rq\rq.
\ACApl\ is strictly weaker than \ATR\ but strictly stronger than \ACApr,
which in turn is strictly stronger than \ACA. The ordinal $\varphi_2(0)$ is
the first fixed point of the $\varepsilon$ function: in \RCA\ we can define a
linear order representing this ordinal, but showing that it is a well-order
requires much stronger theories, since this is the proof-theoretic ordinal of
\ACApl.

\begin{theorem}
\ACApl\ plus \lq\lq $\varphi_2(0)$ is a well-order\rq\rq\ proves the
restriction of Fra\"{\i}ss\'{e}'s conjecture to linear orders of finite Hausdorff
rank, which in turn implies, over \RCA, \ACApr\ plus \lq\lq $\varphi_2(0)$ is
a well-order\rq\rq.
\end{theorem}

\section{A topological version of wqos}\label{sec:topo}

Recall that $Q$ wqo does not imply that $\Pow Q$ with respect to
$\preceq^\flat$ or $\Pf Q$ with respect to $\preceq^\sharp$ is wqo. However
we can still draw some conclusions about these partial orders if we weaken
the conclusion, using a topological notion.

If $(Q,{\preceq})$ is quasi-order we can use $\preceq$ to define a number of
different topologies on $Q$. These include the \emph{Alexandroff topology}
(whose closed sets are the initial intervals of $Q$) and the \emph{upper
topology} (whose basic closed sets are of the form $\set{x \in Q}{\exists y
\in F\, x \preceq y}$ for $F \subseteq Q$ finite). The topological notion
that turns out to be relevant is the following: a topological space is
\emph{Noetherian} if it contains no infinite strictly descending sequences of
closed sets. It turns out that $Q$ is wqo if and only if the Alexandroff
topology on $Q$ is Noetherian, and that $Q$ wqo implies that the upper
topology on $Q$ is Noetherian. Goubault-Larrecq \cite{GL07} proved that $Q$
wqo does imply that the upper topologies of $\Pow Q$ with respect to both
$\preceq^\flat$ and $\preceq^\sharp$ are Noetherian. Frittaion, Hendtlass,
Marcone, Shafer, and Van der Meeren \cite{Noeth} studied these results from
the viewpoint of reverse mathematics, providing along the way proofs that
have a completely different flavor from the category-theoretic arguments used
by Goubault-Larrecq.

Before describing the results from \cite{Noeth} we need to explain the
set-up, which in this case is not obvious because it is necessary to
formalize statements about topological spaces which do not fit in the
frameworks usually considered in subsystems of second-order arithmetic. (If
$Q$ is not an antichain then the Alexandroff and upper topologies are not
$T_1$ and are thus very different from complete separable metric spaces.)
First notice that if the quasi-order $Q$ is countable the Alexandroff and
upper topology can be defined in \RCA\ within the framework of countable
second-countable spaces introduced by Dorais \cite{Dorais}. Expressing the
fact that a countable second-countable space is Noetherian, as well as the
connection mentioned above between $Q$ wqo and the fact that these topologies
are Noetherian are also straightforward in \RCA. However this still does not
suffice to tackle all of Goubault-Larrecq's results, because some of them
deal with topologies defined on the uncountable space $\Pow Q$. To express
that the upper topology of $\Pow Q$ with respect to either $\preceq^\flat$ or
$\preceq^\sharp$ is Noetherian, the authors of \cite{Noeth} devise a way of
representing these topological spaces. This representation shares some
features with other well-established representations of topological spaces,
including the familiar separable complete metric spaces and the countably
based MF spaces introduced by Mummert \cite{Mummert}. In this set-up the main
results are the following (\cite[Theorem 4.7]{Noeth}).

\begin{theorem}[\RCA]\label{Noeth}
The following are equivalent:
\begin{enumerate}[\quad (i)]
  \item \ACA;
  \item if $Q$ is wqo, then the Alexandroff topology of $\Pf Q$ with
      respect to $\preceq^\flat$ is Noetherian;
  \item if $Q$ is wqo, then the upper topology of $\Pf Q$ with respect to
      $\preceq^\flat$ is Noetherian;
  \item if $Q$ is wqo, then the upper topology of $\Pf Q$ with respect to
      $\preceq^\sharp$ is Noetherian;
  \item if $Q$ is wqo, then the upper topology of $\Pow Q$ with respect to
      $\preceq^\flat$ is Noetherian;
  \item if $Q$ is wqo, then the upper topology of $\Pow Q$ with respect to
      $\preceq^\sharp$ is Noetherian.
\end{enumerate}\end{theorem}

In \cite[Section 9.7]{GLbook} Goubault-Larrecq supports his claim that
Noetherian spaces can be thought of as topological versions of wqos, by
proving the following results. Starting from a topological space $X$ he
introduces topologies on $X^{<\N}$ and $\mathcal{T}^X$ and proves the
topological versions of Higman's and Kruskal's theorems, stating that if $X$
is Noetherian then $X^{<\N}$ and $\mathcal{T}^X$ are Noetherian. If $X$ is a
countable second-countable space then so are $X^{<\N}$ and $\mathcal{T}^X$,
which leads to the following so far unexplored question.

\begin{question}
What is the strength of the topological versions of Higman's and Kruskal's
theorems restricted to countable second-countable spaces?
\end{question}

\newcommand{\etalchar}[1]{$^{#1}$}

\end{document}